\documentclass[12pt, oneside]{amsart}

      \usepackage{amssymb,amsthm}
      \usepackage{float}
      \usepackage{amsmath}
      \usepackage{setspace}
      \usepackage{graphicx}

      \theoremstyle{plain}
      \newtheorem{theorem}{Theorem}[section]
      \newtheorem{lemma}[theorem]{Lemma}
      \newtheorem{corollary}[theorem]{Corollary}

      \theoremstyle{definition}
      \newtheorem{definition}[theorem]{Definition}

      \theoremstyle{remark}
      \newtheorem{remark}[theorem]{Remark}
      \newtheorem{example}[theorem]{Example}

      \makeatletter
      \def\@setcopyright{}
      \def\serieslogo@{}
      \makeatother

\begin{document}

   \author{Dinesh Udar and Shiksha Saini}
   \address{Department of Applied Mathematics, Delhi Technological University, Delhi, India}
   \email{dineshudar@yahoo.com}
   \email{shiksha96saini@gmail.com}

  \title[ $\sqrt{J}$-clean rings]{$\sqrt{J}$-clean rings}

\begin{abstract}
In this paper, we study a new class of rings, called $\sqrt{J}$-clean rings. A ring in which every element can be expressed as the addition of an idempotent and an element from $\sqrt{J(R)}$ is called a $\sqrt{J}$-clean ring. Here, $\sqrt{J(R)}=\{ z\in R : z^n\in J(R) \ \mathrm{for \ some} \ n \geq 1  \}$ where, $J(R)$ is the Jacobson radical. We provide the basic properties of $\sqrt{J}$-clean rings. We also show that the class of semiboolean and nil clean rings is a proper subclass of the class of $\sqrt{J}$-clean rings, which itself is a proper subclass of clean rings. We obtain basic properties of $\sqrt{J}$-clean rings and give a characterization of $\sqrt{J}$-clean rings: a ring $R$ is a $\sqrt{J}$-clean ring iff $R/J(R)$ is a $\sqrt{J}$-clean ring and idempotents lift modulo $J(R)$. We also prove that a ring is a uniquely clean ring if and only if it is a uniquely $\sqrt{J}$-clean ring.  Finally, several matrix extensions like $T_n(R)$ and $D_n(R)$ over a $\sqrt{J}$-clean ring are explored.

\end{abstract}


   \date{\today}

   \maketitle

   \section{Introduction}

   Unless and until explicitly mentioned, we consider $R$ as an associative ring with identity $(1\neq 0)$. Let the set of all idempotents and the Jacobson radical be denoted by $Id(R)$ and $J(R)$, respectively. Intensive research has been done on clean rings and their various subclasses. Nicholson, in \cite{liftingidempotentsandexchangerings}, introduced the rings whose every element is $clean$, called clean rings. An element $a\in R$ for which there exists an idempotent $e$ and an invertible element $u$ so that $a=e+u$, is called a clean element. If this decomposition of $a$ commutes, $a$ is a strongly clean element. If the decomposition of $a$ as a clean element is unique, $a$ is called a uniquely clean element. In such cases, the ring $ R$ is known as a strongly clean ring and a uniquely clean ring, respectively.

   Next, in the decomposition of $a$, if the unit is replaced by a nilpotent, we say $a$ is a nil-clean element. The ring in which every element is nil-clean is said to be a nil-clean ring, studied by Diesl in \cite{diesl2013nilcleanrings}. A semiboolean ring or $J$-clean ring is a ring wherein each element is semiboolean; that is, for every $a\in R$, $a=e+j$, where $e\in Id(R)$ and $j\in J(R)$. Another characterization of semiboolean rings is that $R$ is a semiboolean ring if and only if $R/J(R)$ is a boolean ring and idempotents lift modulo $J(R)$. As seen in the case of clean rings, a strongly nil-clean ring (or strongly $J$-clean ring), studied by Chen in \cite{chen2010stronglyJcleanrings}, can be defined. Similarly, a uniquely nil-clean ring (uniquely $J$-clean ring) can also be defined. Both nil-clean rings and semiboolean rings are clean rings.

   Wang and Chen, in \cite{wang2012pseudo}, introduced $\sqrt{J(R)}$ as a subset of ring $R$, defined as $\sqrt{J(R)}=\{ x: x^n \in J(R) \ \mathrm{for \ some } \ n \geq 1\}$. The set $\sqrt{J(R)}$ may not be closed under the binary operation of addition as well as multiplication, and hence, need not be a subring of $R$. Also, $J(R)$ and the nilpotents of $R$ are subsets of $\sqrt{J(R)}$. Motivated by the above developments, and to refine the relations among these rings, we introduce a new class of rings: $\sqrt{J}$-clean rings:
   \begin{definition}
       Any element $x\in R$ is called a $\sqrt{J}$-clean element when it is possible to express it as an addition of an idempotent $e$ and an element $z$ of $\sqrt{J(R)}$ resulting in $x=e+z$. If each element of $R$ is a $\sqrt{J}$-clean element, then $R$ is said to be a $\sqrt{J}$-clean ring.
   \end{definition}
   Examples of $\sqrt{J}$-clean rings are Boolean rings, $2 \times 2$ matrix rings over $\mathbb{F}_2$. 

   We have organized this article as follows: The basic properties of $\sqrt{J}$-clean rings are highlighted, along with providing a characterization of $\sqrt{J}$-clean rings in Section \ref{basic results}. It is shown that semiboolean and nil clean rings are properly contained in the class of $\sqrt{J}$-clean rings, which is properly contained in the class of clean rings. In the next Section, we have studied strongly $\sqrt{J}$-clean rings. In this ring, the decomposition of element $x$ as an addition of an $e\in Id(R)$ and $z\in \sqrt{J(R)}$ satisfies the condition $ez=ze$. We show its relationship with $\sqrt{J}U$ rings, the rings for which $U(R) \subseteq 1+ \sqrt{J(R)}$, which were explored in \cite{sainiczechoslovak}. The characterization of strongly $\sqrt{J}$-clean division rings and semisimple rings is also provided. In Section \ref{uniquely sqrt J clean rings}, we show that the idempotents of a uniquely $\sqrt{J}$-clean ring are central, and their relation with Dedekind finite rings is established. We also show that a ring $R$ is a uniquely clean ring iff it is a uniquely $\sqrt{J}$-clean ring. In Section \ref{matrix rings}, we show that the ring of all $n \times n$ matrices over any $R$ cannot be a strongly $\sqrt{J}$-clean ring. We also discuss the subrings of matrix rings and triangular matrix rings. Finally, we present some conditions under which the ring of all $n\times n$ matrices over $R$ is a $\sqrt{J}$-clean ring.
   
   We will be representing $M_n(R)$, $T_n(R)$, and $D_n(R)$ as the $ n \times n$ matrix ring, upper triangular matrix ring, and upper triangular matrix rings with equal diagonal entries, respectively. Also, $C(R)$, $U(R)$, and $N(R)$ are the center, the group of invertible elements, and the set of nilpotents of $R$, respectively. Also, for any other unexplained term or definition, \cite{lam1991first} can be referred to.
     
   \section{Basic results}\label{basic results}

   We begin this section by highlighting certain basic properties of $\sqrt{J}$-clean rings:

\begin{enumerate}
    \item If $z\in \sqrt{J(R)}$, then $z$ have a $\sqrt{J}$-clean decomposition.
    \item If a $\sqrt{J}$-clean ring is abelian and $J(R)=0$, then $R$ is reduced and therefore, is a boolean ring.
\end{enumerate}

Now, we first list some properties of the elements of $\sqrt{J(R)}$ from \cite{sainiczechoslovak} and \cite{UsqrtJringsshiksha}.
\begin{lemma}\label{ppties of sqrt J(R)} In any ring R, the following holds:
    \begin{enumerate}
    \item For an element $x$ from $\sqrt{J(R)}$ and a central element $y$, $xy \in \sqrt{J(R)}$. The converse holds when $y\in U(C(R))$.
    \item The group of units and the set $\sqrt{J(R)}$ are disjoint.
    \item The set of idempotents and $\sqrt{J(R)}$ have only 0 common, i.e., $Id(R) \cap \sqrt{J(R)} = \{ 0 \}$.
\end{enumerate}
\end{lemma}

\begin{lemma}\label{2 in J(R)}
       If $R$ is a $\sqrt{J}$-clean ring, then $2\in J(R)$.
   \end{lemma}
   \begin{proof}
       Assume $R$ is a $\sqrt{J}$-clean ring. Then, there exists an idempotent $e$ and $z$, an element of $\sqrt{J(R)}$ such that $2=e+z \Rightarrow 1-e=z-1$. Hence, following Lemma \ref{ppties of sqrt J(R)}, we have, $e=0$ and therefore, $2\in \sqrt{J(R)}$. This results in $1-2^ka^k \in U(R)$, resulting in $1-2a \in U(R)$ for any element $a\in R$. Hence, $2\in J(R)$.
   \end{proof}

   \begin{lemma}\label{homorphism of sqrt J clean ring}
       The homomorphic image of a $\sqrt{J}$-clean ring is $\sqrt{J}$-clean.
   \end{lemma}
   \begin{proof}
   The proof is straightforward, as the idempotents and elements of $\sqrt{J(R)}$ are preserved under the homomorphism $\phi$.
   \end{proof}

   \begin{lemma}\label{R1 +R2 is sqrt J clean}
    Let $R_1$ and $R_2$ be two rings. Then $R_1 \times R_2$ is $\sqrt{J}$-clean iff $R_1$ and $R_2$ are $\sqrt{J}$-clean rings.
\end{lemma}
\begin{proof}
    For any two rings $R_1$ and $R_2$, we have $\sqrt{J({R_1 \times R_2})}= \sqrt{J(R_1)} \times \sqrt{J(R_2)}$. Also, $(e_1,e_2)$ is an idempotent in $R_1 \times R_2$ iff $e_1$ and $e_2$ are idempotent in $R_1$ and $R_2$, respectively. Hence, the proof follows from Lemma \ref{homorphism of sqrt J clean ring}.
\end{proof}

It is worthwhile noting that as $J(R)$ and $N(R)$ are the subsets of $\sqrt{J(R)}$, every semiboolean ring (or J-clean ring) and nil-clean ring is a $\sqrt{J}$-clean ring, respectively. Also, if $R$ is a $\sqrt{J}$-clean ring, it is a clean ring also, as presented below:
\begin{lemma}
Every $\sqrt{J}$-clean ring is a clean ring. 
\end{lemma}
\begin{proof}
    If $R$ is a $\sqrt{J}$-clean ring, let $a\in R$. Then for some idempotent $e$ and $z\in \sqrt{J(R)}$, we obtain a $\sqrt{J}$-clean decomposition of $a$ as $a=e+z$. This results in $a=(e-1)+(1+z)$, where $e-1$ is an idempotent. As $-z\in \sqrt{J(R)}$, we have, $1+(-z)=1+z \in U(R)$. Hence, $R$ is a clean ring, as required.
\end{proof}

The above developments leads us to the following observation:
   \begin{equation*}
   \begin{matrix}
        \mathrm{semiboolean \ rings}& \Rightarrow & \sqrt{J}- \mathrm{clean \ rings} & \Rightarrow & \mathrm{clean \ ring}\\
        & & \Uparrow & & \\
        & & \mathrm{nil \ clean \ rings}
   \end{matrix}
   \end{equation*}

Now, we provide examples showing that the above relation is irreversible.
\begin{example}
    \begin{enumerate}
        \item If $R=\mathbb{Z}_9$, then $R$ is a clean ring, as $2\notin J(R)$, using the Lemma \ref{2 in J(R)}, $R$ is not a $\sqrt{J}$-clean ring.
        \item If $R=M_2(\mathbb{Z}_{(2)})$, then $J(R)=M_2(2 \ \mathbb{Z}_{(2)})$ and $R/J(R) \cong M_2(\mathbb{Z}_{2})$. As $\begin{pmatrix}
            0 & 1 \\ 0 & 0 
        \end{pmatrix}$ is a nilpotent matrix in $M_2(\mathbb{Z}_{2})$, $R/J(R)$ is not a boolean ring, and hence, $R$ is not a semiboolean ring. Also, observe that by following \cite[Theorem 3]{breaz2013nilcleanmatrixrings}, $M_2(\mathbb{Z}_{2})$ is a nil-clean ring. This implies $M_2(\mathbb{Z}_{2})$ is a $\sqrt{J}$-clean ring. This gives $R/J(R)$ is a $\sqrt{J}$-clean ring. By \cite[Example 23.2]{lam1991first}, as $\mathbb{Z}_{(2)}$ is a local ring, $R$ is a semi-perfect ring. Now, by following the definition of semi-perfect rings, idempotents of $R/J(R)$ can be lifted 
        to $R$. Hence, by Theorem \ref{characterzation of sqrt J clean rings II}, $R$ is a $\sqrt{J}$-clean ring. 
        \item \cite[Example 2.3]{ukranianstronglypclean} If we let $R=\mathbb{Z}_2 \times \mathbb{Z}_4 \times \mathbb{Z}_8 \times \dots$, then as $(0,2,2, 2 \dots ) \in R$ is not a nil-clean element, $R$ is not a nil-clean ring. Additionally, for every positive integer $k$, $\mathbb{Z}_{2^k}/J(\mathbb{Z}_{2^k}) \cong \mathbb{Z}_2$. Hence, every $\mathbb{Z}_{2^k}$ is a semiboolean ring and therefore, from \cite{nicholson2005cleangeneralrings}, $R$ is a semiboolean ring and hence, $R$ is a $\sqrt{J}$-clean ring.
    \end{enumerate}
\end{example}

Now, we present a characterization of $\sqrt{J}$-clean rings:

\begin{theorem}\label{characterzation of sqrt J clean rings II}
    Let $R$ be any ring. Then, $R$ is a $\sqrt{J}$-clean ring iff $R/J(R)$ is a $\sqrt{J}$-clean ring and idempotents lift modulo $J(R)$. 
\end{theorem}
       \begin{proof}
When $R$ is a $\sqrt{J}$-clean ring, following Lemma \ref{homorphism of sqrt J clean ring}, $R/J(R)$ is a $\sqrt{J}$-clean ring. To prove that idempotents lift modulo $J(R)$, for any element $x\in R$, let $x=f+z$, where $f$ is an idempotent and $z\in \sqrt{J(R)}$. Let $\bar{x} \in R/J(R)$ be an idempotent. Therefore, $\bar{x^2}-\bar{x}=\bar{0} ~ \Rightarrow x-x^2 \in J(R)$. If we are able to show that an idempotent $e$ exists such that $x-e \in J(R)$, then we are done. Observe that $x-x^2=(f+z)-(f+z)^2 ~ \Rightarrow x-x^2 = f+z-f-zf-fz-z^2$. Hence, 
\begin{equation*}
    x-x^2=(1-z)z+(1-z)f-f(1-z)-f2z\in J(R).
\end{equation*}
By letting $1-z=u$, we obtain $uz+uf-fu \in J(R)$. Also, as $z\in \sqrt{J(R)}$, $u = 1-z \in U(R)$. Therefore, $z+f-u^{-1}fu \in J(R) \Rightarrow x-u^{-1}fu \in J(R)$. If $e=u^{-1}fu$, we have $x-e \in J(R)$, as required. 

On the contrary, now let $R/J(R)$ be a $\sqrt{J}$-clean ring and idempotents lift modulo $J(R)$. Hence, for any $\bar{x}\in R/J(R)$, let $\bar{x}=\bar{e}+\bar{z}$, where $\bar{e}\in Id(R/J(R))$ and $\bar{z}\in \sqrt{J(R/J(R))}$. As idempotents lift modulo $J(R)$ and also the elements of $\sqrt{J((R/J(R))}$ lift modulo to $\sqrt{J(R)}$, we can take $x=e+z$, where $e\in Id(R)$ and $z\in \sqrt{J(R)}$. Hence, $R$ is a $\sqrt{J}$-clean ring.
   \end{proof}

\begin{lemma}\label{when does sqrt J(R) implies J(R)}
    For any $\sqrt{J}$-clean ring $R$, if $R$ is abelian, then $J(R)=\sqrt{J(R)}$.
\end{lemma}

\begin{proof}
    Given that $R$ is abelian, $R$ is a Dedekind finite ring. As every $\sqrt{J}$-clean ring is a potent ring, a non-zero idempotent $e$ exists in a right ideal $I \nsubseteq J(R)$. Suppose $z$ is any non-zero element in $\sqrt{J(R)}$ and $I=zR$. Then $e=zr$, where $r\in R$. Then, as $R$ is a Dedekind finite ring, so is $eRe$, and hence, $eze$ is invertible. However, as $z\in \sqrt{J(R)}$, by \cite{sainiczechoslovak}, $eze \in e \sqrt{J(R)} e$, which is not possible.
\end{proof}

\section{strongly $\sqrt{J}$-clean rings}\label{strongly sqrt J clean rings}

We start this Section by defining \textit{strongly} $\sqrt{J}$-clean rings. A $\sqrt{J}$-clean ring in which the decomposition of every element as a sum of an idempotent and an element from $\sqrt{J(R)}$ commutes is called a strongly $\sqrt{J}$-clean ring. Recall that a ring, wherein $U(R) \subseteq 1+ \sqrt{J(R)}$ is called a $\sqrt{J}U$ ring.

\begin{lemma}\label{sqrt J U ring}
    In a ring $R$, a unit u is strongly $\sqrt{J}$-clean if and only if $u$ is expressible as $1+z$, for some element $z\in \sqrt{J(R)}$.
\end{lemma}
\begin{proof}
    Let unit $u$ be a strongly $\sqrt{J}$-clean element satisfying $u=e+z$, where $e$ is an idempotent, $z\in \sqrt{J(R)}$ and $ez=ze$. Then, as $e$ is idempotent, \begin{equation*}
        e=u^2(1-2u^{-1}z)+z^2,
    \end{equation*}
where $v=u^2(1-2u^{-1}z)$ is a unit. On squaring the above equation repeatedly, after some $k$ steps, we obtain $z^{2k}\in J(R)$ such that $e=w+z^{2k}$, where $w$ is invertible. This results in $e$ as a unit. Hence, $e\in U(R) \cap Id(R)$ and thus $e=1$, as required. If $u=1+z$, for some $z\in \sqrt{J(R)}$, then this is a $\sqrt{J}$-clean decomposition of $u$ with $1.z=z.1$ and hence, $u$ is a strongly $\sqrt{J}$-clean ring.
\end{proof}

\begin{lemma}\label{strongly sqrt J clean ring is sqrt JU}
    Every strongly $\sqrt{J}$-clean ring is a $\sqrt{J}U$ ring.
\end{lemma}

\begin{remark}
    If $a\in R$ is strongly $\sqrt{J}$-clean, then $1-a$ is also a strongly $\sqrt{J}$-clean element.
\end{remark}

   \begin{lemma}\label{strongly J clean equals strongly sqrt J clean}
       If a ring R is a strongly $J$-clean ring, then R is a strongly $\sqrt{J}$-clean ring.
   \end{lemma}
   \begin{proof}
       The proof follows clearly as $J(R) \subseteq \sqrt{J(R)}$.
   \end{proof}

   \begin{lemma}
       Every strongly $\sqrt{J}$-clean ring is a strongly clean ring.
   \end{lemma}
   \begin{proof}
       If $R$ is strongly $\sqrt{J}$-clean, let $x$ be any element of $R$. Then, there exists a decomposition $x=e-z$ and $e(-z)=(-z)e$, where, $e\in Id(R)$ and $-z\in \sqrt{J(R)}$. Then $x=(e-1)+(1-z)$, where $(e-1)^2 =e-1$ and $1-z$ is invertible. As $(e-1)(1-z)=(1-z)(e-1)$, $R$ is a strongly clean ring.
   \end{proof}

   \begin{remark}
       As $R[x]$ is never a strongly clean ring, by the above Lemma, $R[x]$ is never a strongly $\sqrt{J}$-clean ring.
   \end{remark}

   \begin{lemma}\label{strongly sqrt J clean division ring is F_2}
       A division ring D is a strongly $\sqrt{J}$-clean ring iff $D\cong \mathbb{F}_2$.
   \end{lemma}
   \begin{proof}
       Let $R$ be any division ring. Then $D\cong \mathbb{F}_2$ follows from $\sqrt{J(D)}=0$ and $Id(D)=\{ 0, 1 \}$. The converse part is straightforward.
   \end{proof}

   \begin{theorem}
       A semisimple ring $R$ is a strongly $\sqrt{J}$-clean ring if and only if $R \cong \mathbb{F}_2 \times \mathbb{F}_2 \times \dots \times \mathbb{F}_2$.
   \end{theorem}
   \begin{proof}
       If $R$ is a semisimple strongly $\sqrt{J}$-clean ring, then from Wedderburn Artin's theorem, we obtain $R\cong \prod M_{n_k}(D_k)$, where $D_k$ is a division ring. Hence, from Lemma \ref{R1 +R2 is sqrt J clean}, $M_{n_k}(D_k)$ is a $\sqrt{J}$-clean for every $k$. Hence, by Lemma \ref{sqrt J U ring}, $M_{n_k}(D_k) \cong D_k$ and therefore, $R \cong \mathbb{F}_2 \times \mathbb{F}_2 \times \dots \times \mathbb{F}_2$. The converse part is evident.
   \end{proof}

   \begin{theorem}
       A strongly $\sqrt{J}$-clean ring $R$ is local if and only if it has no non-trivial idempotents.
   \end{theorem}
   \begin{proof}
       If $R$ is a strongly $\sqrt{J}$-clean local ring, then $R$ has no non-trivial idempotents is evident. For the converse, if $R$ has no non-trivial idempotents, then for every element $a$ in $R$, we have $a \in U(R)$ or $a-1 \in U(R)$. Hence, from Lemma \ref{strongly sqrt J clean ring is sqrt JU}, we obtain $a\in U(R)$ or $a\in \sqrt{J(R)}$. Therefore, we obtain that $R$ is a local ring following \cite[Theorem 2.8]{UsqrtJringsshiksha}.
   \end{proof}

   \begin{lemma}
       For any element $r$ in a ring strongly $\sqrt{J}$-clean ring R, $r^2 - r \in \sqrt{J(R)}$.
   \end{lemma}
   \begin{proof}
       Let $r$ be any element of strongly $\sqrt{J}$-clean ring $R$ such that $r=e+z$, where $e\in Id(R)$ and $z\in \sqrt{J(R)}$ and $ez=ze$. Then, $r^2-r = (e+z)^2 - (e+z)= z-z^2 -2ez$. Following \cite[Corollary 2.6]{UsqrtJringsshiksha} and Lemma \ref{2 in J(R)}, $z-z^2 -2ez \in \sqrt{J(R)}$.
   \end{proof}

\section{uniquely $\sqrt{J}$-clean rings}\label{uniquely sqrt J clean rings}


A uniquely $\sqrt{J}$-clean ring is a $\sqrt{J}$-clean ring in which for every element $a$, the decomposition of $a$ as the sum of an idempotent and an element of $\sqrt{J(R)}$ is unique.
\begin{lemma}\label{idmepotents of uniquely sqrt J clean rings are central}
    The idempotents of a uniquely $\sqrt{J}$-clean ring are central.
\end{lemma}
\begin{proof}
    Assuming $R$ as a uniquely $\sqrt{J}$-clean ring, let $a\in R$ and $i\in Id(R)$. Then, $i+ia(i-1)\in Id(R)$. If $i=i+ia(1-i)$, then as $R$ is uniquely $\sqrt{J}$-clean ring, we have $ia(1-i)=0 \ \Rightarrow ia=iai$. Similarly, as $i+(1-i)ia$ is an idempotent, tracing the above steps gives $ai=iai$. Hence, the idempotents are central.
\end{proof}

\begin{lemma}
    A uniquely $\sqrt{J}$-clean ring R is a Dedekind finite ring.
\end{lemma}
\begin{proof}
The idempotents of a uniquely $\sqrt{J}$-clean ring are central following Lemma \ref{idmepotents of uniquely sqrt J clean rings are central}. Hence, $R$ is an abelian ring and therefore, $R$ is a Dedekind finite ring.
\end{proof}

Now, we proceed to investigate the relationship among uniquely clean rings and uniquely $\sqrt{J}$-clean rings. For that, we prove the following two Lemmas first.

\begin{lemma}\label{uniquely clean iff strongly sqrt J clean with central idempotents}
A ring R is a uniquely clean ring if and only if R is a $\sqrt{J}$-clean ring with central idempotents. 
\end{lemma}
\begin{proof}
    Suppose $R$ is a uniquely clean ring and $x$ is any element in $R$. Following \cite[Theorem 20]{nicholsonandzhouuniquelycleanrings}, a unique idempotent $e$ exists satisfying $x-e \in J(R)$. As $J(R)\subseteq \sqrt{J(R)}$, we obtain that $z=x-e \in \sqrt{J(R)} \Rightarrow \ x=e+z$. Next, the idempotents of $R$ are central following \cite[Lemma 4]{nicholsonandzhouuniquelycleanrings}. This proves $R$ is a $\sqrt{J}$-clean ring.

    If $R$ is a $\sqrt{J}$-clean ring with central idempotents, let $a\in R$. Hence, there exists an idempotent $e$ and $z\in \sqrt{J(R)}$ satisfying $a+1=e+z$ and therefore, $a=e+ (z-1)$. This is a decomposition of $a$ as a clean element. Here, as $z\in \sqrt{J(R)}$, $1-z^n$ is invertible, and hence, $(1-z)(1+z+z^2+ \dots + z^{n-1}) \in U(R) \Rightarrow z-1 \in U(R)$. For uniqueness, we assume $a=e+u$ and $a=f+v$ as two clean representations of $a$, where $e$ and $f$ are idempotents and $u$ and $v$ are units. As its implication, we have $e-f=v-u$. Now, as idempotents are central, by using Lemma \ref{sqrt J U ring}, we obtain $v-1, ~ u-1 \in \sqrt{J(R)}$. Following Lemma \ref{when does sqrt J(R) implies J(R)}, we have $v-1, ~ u-1 \in J(R)$. Hence, $e-f=(v-1)-(u-1) \in \sqrt{J(R)}$ and thus $(e-f)^2 \in \sqrt{J(R)}$. Next, note that as idempotents are central, we have $e-f = (e-f)^3$. This results in $(e-f)^2=(e-f)^4$ and hence, $(e-f)^2$ is an idempotent. Therefore, $(e-f)^2 \in \sqrt{J(R)} \cap Id(R)$. This result in $(e-f)^2=0$ because $Id(R) \cap \sqrt{J(R)} = \{0 \}$. Hence, we get $e-f=0$. Hence, we finally have $e=f$ and $v=u$, and this finally proves $R$ is a uniquely clean ring.
\end{proof}

\begin{lemma}\label{uniquely sqrt J clean ring is strongly sqrt J clean}
    A ring $R$ is a uniquely $\sqrt{J}$-clean ring if and only if $R$ is a $\sqrt{J}$-clean ring with central idempotents. 
\end{lemma}
\begin{proof}
If $R$ is a uniquely $\sqrt{J}$-clean ring, then from Lemma \ref{idmepotents of uniquely sqrt J clean rings are central}, the idempotents of $R$ are central. Hence, $R$ is a $\sqrt{J}$-clean ring with central idempotents. Conversely, if $R$ is a $\sqrt{J}$-clean ring with central idempotents, then, if possible, let there exist two decompositions of any element $a\in R$ such that $a=e+z$ and $a=f+z'$. Here, $e, \ f \in Id(R)$  and $z, z' \in \sqrt{J(R)}$. Hence, $e+z=f+z' \Rightarrow e+(z+1)=f+(z'+1)$. This provides two clean representations of $a+1$ in $R$, which is a uniquely clean ring following Lemma \ref{uniquely clean iff strongly sqrt J clean with central idempotents}. Hence, $e=f$ and $z=z'$ and this proves $R$ is a uniquely $\sqrt{J}$-clean ring.
\end{proof}

So, from Lemma \ref{uniquely clean iff strongly sqrt J clean with central idempotents} and Lemma \ref{uniquely sqrt J clean ring is strongly sqrt J clean}, we get the following result:

\begin{theorem}
    A ring R is a uniquely clean ring iff it is a uniquely $\sqrt{J}$-clean ring.
\end{theorem}

\begin{lemma}\label{local strongly sqrt J clean ring have R/J(R) cong F_2}
       Let $R$ be a ring. Then $R$ is a local and strongly $\sqrt{J}$-clean ring if and only if $R/J(R) \cong \mathbb{F}_2$.
   \end{lemma}
   \begin{proof}
  If $R$ is a local ring, then $R/J(R)$ is a division ring. Let $R$ be a local strongly $\sqrt{J}$-clean ring. Then $R/J(R) \cong \mathbb{F}_2$ following Lemma \ref{strongly sqrt J clean division ring is F_2}. On the contrary, if $R/J(R) \cong \mathbb{F}_2$, then $R/J(R)$ is a strongly $\sqrt{J}$-clean ring. Following \cite[Theorem 15]{nicholsonandzhouuniquelycleanrings}, $R$ is a uniquely clean ring and hence, $R$ is a strongly $\sqrt{J}$-clean ring by Theorem \ref{uniquely clean iff strongly sqrt J clean with central idempotents}
   \end{proof} 

\begin{lemma}
    In a local ring $R$, $R$ is a uniquely $\sqrt{J}$-clean ring if and only if $R$ is a strongly $\sqrt{J}$-clean ring.
\end{lemma}
\begin{proof}
    The proof follows from Lemma \ref{local strongly sqrt J clean ring have R/J(R) cong F_2} and Lemma \ref{uniquely sqrt J clean ring is strongly sqrt J clean}.
\end{proof}
   
   \section{Matrix rings}\label{matrix rings}

When $R$ is a strongly $\sqrt{J}$-clean ring, then by Lemma \ref{strongly sqrt J clean ring is sqrt JU}, $R$ is a $\sqrt{J}U$ ring and hence, $M_n(R)$ is never a strongly $\sqrt{J}$-clean ring following \cite[Theorem 2.13]{sainiczechoslovak}. As a result, we have the following Lemma:

\begin{lemma}
    If R is a strongly $\sqrt{J}$-clean ring, then $M_n(R)$ is $\sqrt{J}$-clean if and only if $n=1$.
\end{lemma}

\begin{lemma}
    If $R$ is a boolean ring, then $M_n(R)$ is a $\sqrt{J}$-clean ring.
\end{lemma}
\begin{proof}
    The proof is clear from \cite[Corollary 6]{breaz2013nilcleanmatrixrings}.
\end{proof}


\begin{lemma}
    For any ring $R$, $R$ is a $\sqrt{J}$-clean ring if and only if $D_n(R)$ is a $\sqrt{J}$-clean ring.
\end{lemma} 
\begin{proof}
    Suppose $R$ is a $\sqrt{J}$-clean ring. Also, let $n=4$ and $a\in R$. Then, there exists an idempotent $e$ and $z\in \sqrt{J(R)}$ satisfying $a=e+z$. Let \begin{equation*}
        A=\begin{pmatrix}
            a & a_1 & a_2 & a_3 \\
            0 & a & a_4 & a_5 \\
            0 & 0 & a & a_6\\
            0 & 0 & 0 & a
        \end{pmatrix}
    \end{equation*} be a matrix in $D_4(R)$.
    Then, \begin{equation*}
        \begin{pmatrix}
            a & a_1 & a_2 & a_3 \\
            0 & a & a_4 & a_5 \\
            0 & 0 & a & a_6\\
            0 & 0 & 0 & a
        \end{pmatrix} = \begin{pmatrix}
            e & 0 & 0 & 0 \\
            0 & e & 0 & 0 \\
            0 & 0 & e & 0\\
            0 & 0 & 0 & e
        \end{pmatrix}+ \begin{pmatrix}
            z & a_1 & a_2 & a_3 \\
            0 & z & a_4 & a_5 \\
            0 & 0 & z & a_6\\
            0 & 0 & 0 & z
        \end{pmatrix}.
    \end{equation*}
    In this decomposition, $\begin{pmatrix}
            e & 0 & 0 & 0 \\
            0 & e & 0 & 0 \\
            0 & 0 & e & 0\\
            0 & 0 & 0 & e
        \end{pmatrix}$ is an idempotent in $D_4(R)$ and as $z \in \sqrt{J(R)}$, for some $m \geq 1$, $z^m \in J(R)$. This results in $\begin{pmatrix}
            z & a_1 & a_2 & a_3 \\
            0 & z & a_4 & a_5 \\
            0 & 0 & z & a_6\\
            0 & 0 & 0 & z
        \end{pmatrix}^m \in J(D_4(R))$ and hence, $\begin{pmatrix}
            z & a_1 & a_2 & a_3 \\
            0 & z & a_4 & a_5 \\
            0 & 0 & z & a_6\\
            0 & 0 & 0 & z
        \end{pmatrix} \in \sqrt{J(D_4(R))}$. Hence, we have obtained a $\sqrt{J}$-clean decomposition of $A$ in $D_4(R)$, as required.
    
    Now, assume $D_4(R)$ is a $\sqrt{J}$-clean ring and $a\in R$. Also, any idempotent of $D_4(R)$ will attain the form $\begin{pmatrix}
        e & 0 & 0 & 0\\ 0 & e & 0 & 0\\ 0 & 0 & e & 0\\ 0 & 0 & 0 & e
    \end{pmatrix}$, where $e\in Id(R)$. Now, if $A=\begin{pmatrix}
            a & a_1 & a_2 & a_3 \\
            0 & a & a_4 & a_5 \\
            0 & 0 & a & a_6\\
            0 & 0 & 0 & a
        \end{pmatrix}$, then let $E=\begin{pmatrix}
        e & 0 & 0 & 0\\ 0 & e & 0 & 0\\ 0 & 0 & e & 0\\ 0 & 0 & 0 & e
    \end{pmatrix} \in Id(D_4(R))$ and $Z=\begin{pmatrix}
            z & a_1 & a_2 & a_3 \\
            0 & z & a_4 & a_5 \\
            0 & 0 & z & a_6\\
            0 & 0 & 0 & z
        \end{pmatrix} \in \sqrt{J(D_4(R))}$ such that \begin{equation*}
            \begin{pmatrix}
            a & a_1 & a_2 & a_3 \\
            0 & a & a_4 & a_5 \\
            0 & 0 & a & a_6\\
            0 & 0 & 0 & a
        \end{pmatrix} = \begin{pmatrix}
            e & a_1 & a_2 & a_3 \\
            0 & e & a_4 & a_5 \\
            0 & 0 & e & a_6\\
            0 & 0 & 0 & e
        \end{pmatrix}+ \begin{pmatrix}
            z & a_1 & a_2 & a_3 \\
            0 & z & a_4 & a_5 \\
            0 & 0 & z & a_6\\
            0 & 0 & 0 & z
        \end{pmatrix}.
        \end{equation*}
        This provides a $\sqrt{J}$-clean decomposition of $a$ satisfying $a=e+z$, where $e\in Id(R)$ and $z\in \sqrt{J(R)}$.  
\end{proof}

\begin{theorem}\label{Mn(F2) is sqrt J clean ring}
    The following are equivalent in a field S:
    \begin{enumerate}
        \item S is isomorphic to $\mathbb{F}_2$.
        \item for every positive integer n, $M_n(S)$ is a $\sqrt{J}$-clean ring.
        \item for some positive integer n, $M_n(S)$ is a $\sqrt{J}$-clean ring.
    \end{enumerate}
\end{theorem}
\begin{proof}
If $S \cong \mathbb{F}_2$, then by \cite{breaz2013nilcleanmatrixrings}, $M_n(S)$ is a nil-clean ring for every positive integer $n$. This results in $M_n(S)$ being a $\sqrt{J}$-clean ring for every positive integer $n$, proving (1) $\Rightarrow$ (2).\\
Now, (2) $\Rightarrow$ (3) is evident.\\
Let $M_n(S)$ be a $\sqrt{J}$-clean ring for some positive integer $n$. Observe that $2I_n$ is a central element, and this gives $char(S)=2$. If $r$ is any non-zero element in $R$, we have $rI_n$ is a unit in a $\sqrt{J}$-clean ring $M_n(S)$. Hence, following Lemma \ref{sqrt J U ring}, $rI_n$ is a $\sqrt{J}U$ element. This gives $rI_n=I_n+ Z$, for some $Z\in \sqrt{J(M_n(S))}$ and hence, $r=1+z$, for some $z\in \sqrt{J(R)}$. Hence, we have $r=1$ and therefore, $S \cong \mathbb{F}_2$, proving (3) $\Rightarrow$ (1).
\end{proof}

Next, we extend the above result to any division ring $D$.

\begin{theorem}Let $D$ be a division ring and $n \geq 1$. Then $M_n(D)$ is a $\sqrt{J}$-clean ring if and only if $D \cong \mathbb{F}_2$. \end{theorem}
\begin{proof}
If $D \cong \mathbb{F}_2$, then by Theorem \ref{Mn(F2) is sqrt J clean ring}, $M_n(\mathbb{F}_2)$ is a $\sqrt{J}$-clean ring. On the other hand, if $M_n(D)$ is a division ring, from \cite{UsqrtJringsshiksha}, $N(M_n(D))= \sqrt{J(M_n(D))}.$ Hence, if $D$ is a division ring, $M_n(D)$ is a $\sqrt{J}$-clean ring, if and only if it is a nil-clean ring. The proof further follows from \cite[Theorem 3]{kocsan2014matrixringoverdivisionringisnilclean}. 
\end{proof}

Recollect that \begin{equation*}T(R,M)=\left \{  \begin{pmatrix}
    a & m \\ 0 & a
\end{pmatrix} : a \in R , m \in M \right \}\end{equation*} is a subring of \begin{equation*} T(R,R,M)=\left \{ \begin{pmatrix}
    a & M \\ 0 & b
\end{pmatrix} : a, b \in R ; \ m\in M \right \}.\end{equation*} Also, $T(R,M)$ is equivalent to the trivial extension of $R$ and $M$, i.e., $R\propto M =\{ (a,m): a\in R , m\in M \}$. It forms a ring with operations component-wise addition and multiplication defined by \begin{equation*} (r_1,m_1)(r_2,m_2)=(r_1r_2, r_1m_2+m_1r_2). \end{equation*} Also, $T(R,R) \cong R[x]/(x^2)$ and \begin{equation*}
    \sqrt{J(T(R,M))}=\{ (z,m) : z\in \sqrt{J(R)}, \ m\in M \}.
\end{equation*}

\begin{lemma}\label{T(R,M) is sqrt J -clean}
For a ring $R$ and bimodule $M$ over it, $T(R,M)$ is a $\sqrt{J}$-clean ring if and only if $R$ is a $\sqrt{J}$-clean ring.
\end{lemma}
\begin{proof}
    If $R$ is a $\sqrt{J}$-clean ring, let $\begin{pmatrix}
        a & m \\ 0 & a
    \end{pmatrix} \in T(R,M)$, where $a$ and $m$ are elements of $R$ and $M$, respectively. As $a\in R$, let the $\sqrt{J}$-clean decomposition of $a$ be $a=e+z$, where $e$ is an idempotent and $z\in \sqrt{J(R)}$. This results in
    \begin{equation*}
        \begin{pmatrix}
            a & m \\ 0 & a
        \end{pmatrix}= \begin{pmatrix}
            e & 0 \\ 0 & e
        \end{pmatrix} + \begin{pmatrix}
            z & m \\ 0 & z
        \end{pmatrix},
    \end{equation*} which is a $\sqrt{J}$-clean decomposition of $\begin{pmatrix}
        a & m \\ 0 & a
    \end{pmatrix}$. Hence, $T(R,M)$ is a $\sqrt{J}$-clean ring.

    Let $T(R,M)$ be a $\sqrt{J}$-clean ring. Then, for any $\begin{pmatrix}
        a & m\\ 0 & a
    \end{pmatrix} \in T(R,M)$, 
    \begin{equation*}
        \begin{pmatrix}
            a & m \\ 0 & a
        \end{pmatrix}= \begin{pmatrix}
            e & m_1 \\ 0 & e
        \end{pmatrix} + \begin{pmatrix}
            z & m_2 \\ 0 & z
        \end{pmatrix},
    \end{equation*}
    where, $\begin{pmatrix}
            e & m_1 \\ 0 & e
        \end{pmatrix}$ is an idempotent and $\begin{pmatrix}
            z & m_2 \\ 0 & z
        \end{pmatrix} \in \sqrt{J(T(R,M))}$. As a result, we have, $e^2=e$ and $z\in \sqrt{J(R)}$ satisfying $a=e+z$. Hence, $R$ is a $\sqrt{J}$-clean ring.
\end{proof}

\begin{lemma}
    Let $R$ be any ring. Then $R \propto R$ is a $\sqrt{J}$-clean ring iff $T(R,R)$ is a $\sqrt{J}$-clean ring iff $R[x]/(x^2)$ is a $\sqrt{J}$-clean ring.
\end{lemma}

 Next, we recollect the Morita context. We define $(m,n) \mapsto mn$ for the context product $M \times N \to X$ and $(n,m) \mapsto nm$ for the context product $N \times M \to Y$. We let $_XM_Y$ and $_YN_X$ as bimodules over the rings $X$ and $Y$. Then a 4-tuple ring $\begin{pmatrix}
       X & M \\ N & Y
   \end{pmatrix}$ is called a Morita context, with usual matrix operations. If $MN$ and $NM$ both are 0, i.e., the context products are trivial, then the Morita context is called a trivial Morita context or a Morita context with zero pairings, as seen in \cite{haghany1999hopficity}. Also, \begin{equation*}
       \begin{pmatrix}
       X & M \\ N & Y
   \end{pmatrix} \cong T(X \times Y, M \oplus N).
   \end{equation*} with trivial morita context $\begin{pmatrix}
       X & M \\ N & Y
   \end{pmatrix}$. Additionally, \cite{kocsan2015pp, marianne1987rings} can be seen for further study on the Morita context. Examples include formal triangular matrices and $T_n(R)$.

\begin{theorem}
    Let M be a bimodule over the rings $R_1$ and $R_2$. Then $\begin{pmatrix}
        R_1 & M \\ 0 & R_2
    \end{pmatrix}$ is a $\sqrt{J}$-clean ring if and only if $R_1$ and $R_2$ are $\sqrt{J}$-clean rings.
\end{theorem}

\begin{proof}
    Let $\begin{pmatrix}
        R_1 & M \\ 0 & R_2
    \end{pmatrix}$ be a $\sqrt{J}$-clean ring. Hence, $T(R_1 \times R_2 , M)$ is a $\sqrt{J}$-clean ring. Following Lemma \ref{T(R,M) is sqrt J -clean}, this gives $R_1 \times R_2$ is a $\sqrt{J}$-clean ring, which results in $R_1$ and $R_2$ as a $\sqrt{J}$-clean ring. 

    Now, let $R_1$ and $R_2$ be $\sqrt{J}$-clean rings. Hence, following Lemma \ref{R1 +R2 is sqrt J clean} and Lemma \ref{T(R,M) is sqrt J -clean}, we obtain that $T(R_1 \times R_2, M)$ is a $\sqrt{J}$-clean ring. This proves $\begin{pmatrix}
        R_1 & M \\ 0 &R_2
    \end{pmatrix}$ is $\sqrt{J}$-clean.
\end{proof}

\begin{corollary}
    If $n\geq 2$, if $T_n(R)$ is a $\sqrt{J}$-clean ring, then $R$ is also a $\sqrt{J}$-clean ring.
\end{corollary}

Let $R$ be a ring $s\in C(R)$. Define $K_s (R)=\begin{pmatrix}
    R & R\\
    R & R
\end{pmatrix}$. Then $K_s(R)$ forms a ring with the component-wise addition and multiplication defined as $\begin{pmatrix}
    a_1 & x_1 \\
    y_1 & b_1
\end{pmatrix} \begin{pmatrix}
    a_2 & x_2 \\
    y_2 & b_2
\end{pmatrix}=\begin{pmatrix}
    a_1 a_2 +sx_1y_2 & a_1x_2+x_1b_2\\
    y_1a_2+b_1y_1 & sy_1x_2 + b_1b_2
\end{pmatrix}$. Here $s$ is referred to as the multiplier of $K_s(R)$. In fact, $K_s(R)$ is a special kind of Morita context. When $R=P=Q=N=M$ in a Morita context $\begin{pmatrix}
    P & M\\
    N & Q
\end{pmatrix}$, it is referred to as the generalized matrix ring over the ring $R$.

\begin{lemma}
    Let R be a ring and $s\in J(R)$. If $K_s(R)$ is a $\sqrt{J}$-clean ring, then $R$ is a $\sqrt{J}$-clean ring.
\end{lemma}

\begin{proof}
   If $K_s(R)$ is a $\sqrt{J}$-clean ring, let $\begin{pmatrix}
    a & 0 \\ 0 & 0
\end{pmatrix} \in K_s(R)$. Then there exists an idempotent $\begin{pmatrix}
    e & 0 \\ 0 &0 
\end{pmatrix}$ and $\begin{pmatrix}
    z & 0 \\ 0 & 0
\end{pmatrix} \in \sqrt{J(K_s(R))}$ satisfying
\begin{equation*}
    \begin{pmatrix}
    a & 0 \\ 0 & 0
\end{pmatrix} = \begin{pmatrix}
    e & 0 \\ 0 &0 
\end{pmatrix} + \begin{pmatrix}
    z & 0 \\ 0 & 0
\end{pmatrix}. \end{equation*} 
This gives $e\in Id(R)$ and $z\in \sqrt{J(R)}$ and hence, we get, for any element $a$ in $R$, $a=e+z$ and this proves $R$ is a $\sqrt{J}$-clean ring.
\end{proof}

Let $R$ be a ring, $s\in C(R)$ and $n \geq 2$. Then $M_n(R;s)$ denotes the formal matrix ring over $R$ defined by $s$. It is the ring of all $n\times n$ matrices under usual addition and multiplication defined as $(a_{ij}) (b_{ij})=(c_{ij})$, where $(a_{ij})$, $(b_{ij})$ are the $n \times n$ matrices with entries from $R$. Here, $c_{ij}=\sum_{k=1}^n \delta^{\delta_{ikj}}a_{ik}b_{kj}$, where $\delta_{ikj}=1+\delta_{ik}-\delta_{ij}-\delta_{jk}$ with $\delta$ representing the Kronecker's delta function. It is evident that if $n=1$, then $M_n (R;s)$ is the ring $R$.

\begin{lemma}
    \cite[Proposition 11, 32]{tang2013classofformalmatrixrings} Let $R$ be a ring. Then 
    \begin{enumerate}
        \item If $A$ is any matrix from $M_n(R;s)$, then $A$ is a unit if and only if $\det_s (A)$ is invertible.
        \item If $s\in Z(R)$, then 
        \begin{equation*}
            J(M_n(R;s))=\begin{pmatrix}
                J_s(R) & J_s(R) & \dots & J_s(R)\\ J_s(R) & J_s(R) & \dots & J_s(R) \\ \vdots & \vdots & & \vdots \\ J_s(R) & J_s(R) & \dots & J_s(R)
            \end{pmatrix}.
        \end{equation*}
    \end{enumerate}
\end{lemma}

\begin{theorem}
    In a ring $R$, let $s\in Z(R) \cap J(R)$. If $M_n(R[[x]]/(x^m);s)$ is a $\sqrt{J}$-clean ring, then R is a $\sqrt{J}$-clean ring.
\end{theorem}
\begin{proof}
    If $R[[x]]/(x^m)$ is a $\sqrt{J}$-clean ring, then on considering a ring epimorphism $\alpha : R[[x]]/(x^m)\to R$ given by $\alpha(\bar{p})=p(0)$, we obtain that $R$ is a $\sqrt{J}$-clean ring. So, the proof reduces to proving $R[[x]]/(x^m)$ is a $\sqrt{J}$-clean ring and for that, it is sufficient to show that if $M_n(R;s)$ is a $\sqrt{J}$-clean ring, then $R$ is a $\sqrt{J}$-clean ring. As $M_n(R;s)$ is a $\sqrt{J}$-clean rinf, then let $\begin{pmatrix}
        a & 0 \\ 0 & 0
    \end{pmatrix} \in M_n(R;s)$ possesses a $\sqrt{J}$-clean decomposition given by: \begin{equation*}
        \begin{pmatrix}
            a & 0 \\ 0 & 0
        \end{pmatrix} = \begin{pmatrix}
            e & 0 \\ 0 & 0
        \end{pmatrix} + \begin{pmatrix}
            z & 0 \\ 0 & 0
        \end{pmatrix},
    \end{equation*} where $\begin{pmatrix}
            e & 0 \\ 0 & 0
        \end{pmatrix}$ is an idempotent and $\begin{pmatrix}
            z & 0 \\ 0 & 0
        \end{pmatrix} \in \sqrt{J(M_n(R;s))}$. As a result, we get $e^2=e$ and $z\in \sqrt{J(R)}.$ This provides a $\sqrt{J}$-clean decomposition of $a$, given by $a=e+z$ and hence, $R$ is a $\sqrt{J}$-clean ring.
\end{proof}

\end{document}